\documentclass[12pt, reqno]{amsart}
\usepackage{amsmath, amsthm, amscd, amsfonts, amssymb, graphicx, color}
\usepackage[bookmarksnumbered, colorlinks, plainpages]{hyperref}
\hypersetup{colorlinks=true,linkcolor=red, anchorcolor=green, citecolor=cyan, urlcolor=red, filecolor=magenta, pdftoolbar=true}

\newtheorem{theorem}{Theorem}[section]
\newtheorem{lemma}[theorem]{Lemma}

\newtheorem{cor}[theorem]{Corollary}
\theoremstyle{definition}
\newtheorem{definition}[theorem]{Definition}

\theoremstyle{remark}
\newtheorem{remark}[theorem]{\bf{Remark}}
\numberwithin{equation}{section}
\begin{document}

\title[Refinement of   seminorm and numerical radius inequalities]  {\Small{Refinement of  seminorm and numerical radius inequalities  of semi-Hilbertian space operators}}

\author{  Pintu Bhunia, Kallol Paul and Raj Kumar Nayak }

\address{(Bhunia)Department of Mathematics, Jadavpur University, Kolkata 700032, India}
\email{pintubhunia5206@gmail.com}

\address{(Paul)Department of Mathematics, Jadavpur University, Kolkata 700032, India}
\email{kalloldada@gmail.com}

\address{(Nayak)
Department of Mathematics, Jadavpur University,
Kolkata 700032, West Bengal, India}
\email{rajkumarju51@gmail.com}



\thanks{First and third author would like to thank UGC, Govt. of India for the financial support in the form of  research fellowship. Prof. Kallol Paul would like to thank RUSA 2.0, Jadavpur University for the partial support.}




\date{}
\maketitle
\begin{abstract}
We give new inequalities for $A$-operator seminorm and $A$-numerical radius of semi-Hilbertian space operators  and show that the inequalities obtained  here generalize and improve on the existing ones. Considering a complex Hilbert space $\mathcal{H}$ and a non-zero positive bounded linear operator $A$ on $\mathcal{H},$  we show with among other seminorm inequalities, if  $S,T,X\in \mathcal{B}_A(\mathcal{H})$, i.e., if $A$-adjoint of $S,T,X$ exist then 
$$2\|S^{\sharp_A}XT\|_A \leq \|SS^{\sharp_A}X+XTT^{\sharp_A}\|_A.$$
Further, we prove that if $T\in \mathcal{B}_A(\mathcal{H})$ then
\begin{eqnarray*}
\frac{1}{4}\|T^{\sharp_{A}}T+TT^{\sharp_{A}}\|_A \leq  \frac{1}{8}\bigg( \|T+T^{\sharp_{A}}\|_A^2+\|T-T^{\sharp_{A}}\|_A^2\bigg), ~~\textit{and}
\end{eqnarray*}
\begin{eqnarray*}
\frac{1}{8}\bigg( \|T+T^{\sharp_{A}}\|_A^2+\|T-T^{\sharp_{A}}\|_A^2\bigg) +\frac{1}{8}c_A^2\big(T+T^{\sharp_{A}}\big)+\frac{1}{8}c_A^2\big(T-T^{\sharp_{A}}\big)
\leq w^2_A(T).
\end{eqnarray*}
Here $w_A(.), c_A(.)$  and $\|.\|_A $ denote $A$-numerical radius, $A$-Crawford number and $A$-operator seminorm, respectively.

\end{abstract}

\section{Introduction}

\noindent
Let $\mathcal{B}(\mathcal{H})$ denote the $C^{\ast}$-algebra of all bounded linear operators acting on a complex Hilbert space $\mathcal{H}$ with inner product $\langle \cdot , \cdot \rangle$ and $\|\cdot\| $ is the norm induced from the inner product $\langle \cdot , \cdot \rangle$.  For $T \in \mathcal{B}(\mathcal{H})$, let $\|T\|,$ $w(T)$ and $c(T)$ denote the operator norm, the numerical radius and the Crawford number of $T$, respectively. Note that 
$$w(T)=\sup_{\substack{x\in \mathcal{H},\\ \|x\|= 1}} |\langle Tx,x\rangle | ~~\mbox{and}~~ c(T)=\inf_{\substack{x\in \mathcal{H},\\ \|x\|= 1}} |\langle Tx,x\rangle |.$$
The range space and the null space of $T$ are denoted by $\mathcal{R}(T)$ and  $\mathcal{N}(T)$, respectively. Let $T^*$ be the adjoint of $T$ and $|T|=(T^*T)^{\frac{1}{2}}$. 
We reserve the letter $A$ for a non-zero positive bounded linear operator on $\mathcal{H}$ and so  $\langle Ax , x\rangle\geq 0,\;\forall\;x\in \mathcal{H}.$ Consider the semi-inner product $\langle \cdot, \cdot \rangle_A$ on $\mathcal{H}$ induced by $A$, namely, $$\langle x , y\rangle_{A} =\langle Ax , y\rangle,~~\forall\;x,y\in \mathcal{H}.$$ 
 The semi-inner product $\langle \cdot, \cdot \rangle_A$ induces a seminorm $\| \cdot \|_A$  on $\mathcal{H}$   given by $\|x\|_A=\sqrt{\langle x,x\rangle_A}$, $\forall~~x\in \mathcal{H}$. This makes $\mathcal{H}$ a semi-Hilbertian space. One can verify that $\|\cdot\|_A$ is a norm on $\mathcal{H}$ if and only if $A$ is injective. And $(\mathcal{H},\|\cdot\|_A)$ is complete if and only if $\mathcal{R}(A)$ is closed in $\mathcal{H}$. An operator $ T \in \mathcal{B}(\mathcal{H}) $ is said to be $A$-bounded if there exists $c >0$ such that $ \|Tx\|_A \leq c \|x\|_A, \forall x \in \mathcal{H}.$ Now we define the $A$-adjoint operator.
\begin{definition} 
An operator $S\in\mathcal{B}(\mathcal{H})$ is said to be an $A$-adjoint of $T \in \mathcal{B}(\mathcal{H})$  if for every $x,y\in \mathcal{H}$, the identity $\langle Tx , y\rangle_A=\langle x , Sy\rangle_A$ holds, i.e., $S$ is a solution of the equation $AX=T^*A$.
\end{definition}
\noindent The existence of an $A$-adjoint operator is not guaranteed. The set of all operators which admit $A$-adjoint is denoted by $\mathcal{B}_{A}(\mathcal{H})$. By Douglas Theorem \cite{doug}, we  have
\begin{align*}
\mathcal{B}_{A}(\mathcal{H}) = \left\{T\in \mathcal{B}(\mathcal{H})\,:\;T^{*}\left(\mathcal{R}(A)\right)\subseteq \mathcal{R}(A)\right\}.
\end{align*}
 If $T\in \mathcal{B}_A(\mathcal{H})$, the reduced solution of the equation $AX=T^*A$ is a distinguished $A$-adjoint operator of $T$, which is denoted by $T^{\sharp_A}$. Note that, $T^{\sharp_A}=A^\dag T^*A$ where $A^\dag$ is the Moore-Penrose inverse of $A$. 
Again, by applying Douglas theorem, it can observed that
\begin{align*}
\mathcal{B}_{A^{1/2}}(\mathcal{H}) = \left\{T\in \mathcal{B}(\mathcal{H})\,:\;T^{*}\left(\mathcal{R}(A^{1/2})\right)\subseteq \mathcal{R}(A^{1/2})\right\}.
\end{align*}
It is easy to check that the collection of all $A$-bounded operators in $\mathcal{B}(\mathcal{H})$ is $ \mathcal{B}_{A^{1/2}}(\mathcal{H})$, i.e., 
 $$\mathcal{B}_{A^{1/2}}(\mathcal{H}) = \{T \in \mathcal{B}(\mathcal{H}) : \exists~~ \lambda > 0 ~~\mbox{such that}~~  \|Tx\|_A \leq \lambda \|x\|_A, ~~\forall ~~x \in \mathcal{H}\}.$$ 
We note that $\mathcal{B}_{A}(\mathcal{H})$ and $\mathcal{B}_{A^{1/2}}(\mathcal{H})$ are two subalgebras of $\mathcal{B}(\mathcal{H})$. Moreover, the following inclusion holds $$\mathcal{B}_{A}(\mathcal{H})\subseteq\mathcal{B}_{A^{1/2}}(\mathcal{H})\subseteq
\mathcal{B}(\mathcal{H}).$$  
Further, the semi-inner product induces the following $A$-operator seminorm on $\mathcal{B}_{A^{1/2}}(\mathcal{H})$,
\begin{equation*}
\|T\|_A=\sup_{\substack{x\in \overline{\mathcal{R}(A)},\\ x\not=0}}\frac{\|Tx\|_A}{\|x\|_A}=\sup_{\substack{x\in \mathcal{H},\\ \|x\|_{A}= 1}} \|Tx\|_{A}.
\end{equation*}
For $T \in \mathcal{B}_{A^{1/2}}(\mathcal{H})$, the $A$-numerical radius and the $A$-Crawford number of $T$, denoted as $w_A(T)$ and $c_A(T)$, are defined respectively as 
\[ w_A(T)=\sup_{\substack{x\in \mathcal{H},\\ \|x\|_{A}= 1}} |\langle Tx , x\rangle_A|$$ and $$c_A(T)=\inf_{\substack{x\in \mathcal{H},\\ \|x\|_{A}= 1}} |\langle Tx , x\rangle_A|. \]
Here we note that if we consider $A=I$ then $\|T\|_A=\|T\|$, $w_A(T)=w(T)$ and $c_A(T)=c(T)$.
It is well-known that if $T\in \mathcal{B}_{A}(\mathcal{H})$ then the following inequality holds
\begin{eqnarray}\label{1.1}
\frac{\|T\|_A}{2}\leq w_A(T) \leq \|T\|_A.
\end{eqnarray}
Zamani in \cite{Z} improved on the inequality (\ref{1.1}) to prove that
\begin{eqnarray}\label{1.2}
\frac{1}{4}  \|T^{\sharp_{A}}T+TT^{\sharp_{A}} \|_A\leq w^2_A(T) \leq \frac{1}{2}\|T^{\sharp_{A}}T+TT^{\sharp_{A}}\|_A.
\end{eqnarray}
Various other refinements of (\ref{1.1}) have also been obtained, we refer the interested readers to \cite {BPN, BNP, BP2,BFP, BNP2, kais01}. It is useful to recall that $T\in \mathcal{B}(\mathcal{H})$ is said to be $A$-selfadjoint if $AT$ selfadjoint, i.e., $AT=T^*A$  and it is called $A$-positive if $AT$ positive. It is clear that if $T\in \mathcal{B}(\mathcal{H})$ is $A$-selfadjoint then $T\in \mathcal{B}_A(\mathcal{H})$. It is well-known that if $T\in \mathcal{B}(\mathcal{H})$ is $A$-selfadjoint then $\|T\|_A=w_A(T)$. Here we also note that if $S,T\in \mathcal{B}_A(\mathcal{H})$ then $(ST)^{\sharp_A}=T^{\sharp_A}S^{\sharp_A}$, $\|ST\|_A\leq \|S\|_A\|T\|_A$ and $\|Sx\|_A\leq \|S\|_A\|x\|_A$ for all $x\in \mathcal{H}$. For more information related to $A$-adjoint operators, we refer to \cite{Arias1}.

\smallskip
In this paper, we develop many $A$-operator seminorm and $A$-numerical radius inequalities for an operator  $T\in \mathcal{B}_A(\mathcal{H})$ improving the existing inequalities.
In particular, we show that if $S,T,X\in \mathcal{B}_A(\mathcal{H})$, then
$$2 \|S^{\sharp_A}XT\|_A \leq \|SS^{\sharp_A}X+XTT^{\sharp_A}\|_A.$$
Further, we obtain that if $T\in \mathcal{B}_A(\mathcal{H})$, then
\begin{eqnarray*}
\frac{1}{4}\|T^{\sharp_{A}}T+TT^{\sharp_{A}}\|_A+\frac{1}{8}c_A^2\big(T+T^{\sharp_{A}}\big)+\frac{1}{8}c_A^2\big(T-T^{\sharp_{A}})\leq w^2_A(T),
\end{eqnarray*}
\begin{eqnarray*}
\frac{1}{4}\|T^{\sharp_{A}}T+TT^{\sharp_{A}}\|_A &\leq & \frac{1}{4\sqrt{2}} \bigg(\big\|T+T^{\sharp_{A}}\big\|_A^4+\big \|T-T^{\sharp_{A}}\big\|_A^4\bigg)^{\frac{1}{2}} \leq w^2_A(T).
\end{eqnarray*}
 We also prove that if $T\in \mathcal{B}_A(\mathcal{H})$ and $A(T+T^{\sharp_{A}})^2 (T-T^{\sharp_{A}})^2=0,$  then $$w^2_A(T)  = 	\frac{1}{2}\|T^{\sharp_{A}}T+TT^{\sharp_{A}}\|_A.$$ 
 
 \smallskip

\noindent The technique that we use  in developing the inequalities is little different from the one used in earlier works like \cite{BPN,MXZ,Z}, we briefly discuss it here.  The semi-inner product $\langle\cdot ,\cdot\rangle_A$ induces an inner product on the quotient space $\mathcal{H}/\mathcal{N}(A)$ defined as
$$[\overline{x},\overline{y}]=\langle Ax , y\rangle,$$
for all $\overline{x}= x + \mathcal{N}(A),\overline{y}= y + \mathcal{N}(A) \in \mathcal{H}/\mathcal{N}(A)$. Note that $(\mathcal{H}/\mathcal{N}(A),[\cdot,\cdot])$ is not complete unless $\mathcal{R}(A)$ is closed in $\mathcal{H}$.  L. de Branges and J. Rovnyak  \cite{branrov} showed that the completion of $\mathcal{H}/\mathcal{N}(A)$    is isometrically isomorphic to the Hilbert space $\mathcal{R}(A^{1/2})$
with the inner product
$$(A^{1/2}x,A^{1/2}y)=\langle P_{\overline{\mathcal{R}(A)}}x , P_{\overline{\mathcal{R}(A)}}y\rangle,\;\forall\, x,y \in \mathcal{H}.$$ Here $P_{\overline{\mathcal{R}(A)}}$ denotes the projection onto $\overline{\mathcal{R}(A)}$. 
The Hilbert space $\left(\mathcal{R}(A^{1/2}), (\cdot,\cdot)\right)$ is denoted by $\mathbf{R}(A^{1/2})$ and we use the symbol
$\|\cdot\|_{\mathbf{R}(A^{1/2})}$ to represent the norm induced by the inner product $(\cdot,\cdot)$.  For more information related to the Hilbert space $\mathbf{R}(A^{1/2})$, we refer the interested readers to \cite{acg3}. Note that the fact $\mathcal{R}(A)\subseteq \mathcal{R}(A^{1/2}) $ implies that $(Ax,Ay)=\langle x  , y\rangle_A$.
This implies the useful relation $$\|Ax\|_{\mathbf{R}(A^{1/2})}=\|x\|_A,\;\forall\,x\in \mathcal{H}.$$  To proceed further we need the following lemma which gives a nice connection between  $T\in \mathcal{B}_{A^{1/2}}(\mathcal{H})$ and $\widetilde{T}\in \mathcal{B}(\mathbf{R}(A^{1/2}))$. 

\begin{lemma} $($\cite[Prop. 3.6]{acg3}$)$  \label{lem1}
Let $T\in \mathcal{B}(\mathcal{H})$ and let $Z_{A}: \mathcal{H} \rightarrow \mathbf{R}(A^{1/2})$ be defined by $Z_{A}x = Ax,~~\forall~~x\in \mathcal{H}$. Then $T\in \mathcal{B}_{A^{1/2}}(\mathcal{H})$ if and only if there exists unique $\widetilde{T}\in \mathcal{B}(\mathbf{R}(A^{1/2}))$ such that $Z_AT =\widetilde{T}Z_A$.
\end{lemma}

\noindent There are many important well-known relations between $T$ and $\widetilde{T}$ , we mention a few of them  in the form of the following lemma.
\begin{lemma}\label{lem2}$($\cite[Prop. 2.9]{majsecesuci}$)$
	Let $T\in \mathcal{B}_A(\mathcal{H})$. Then\\
	$(i)~\widetilde{T^{\sharp_A}}=\big(\widetilde{T}\big)^*\;\text{ and }\; \widetilde{({T^{\sharp_A}})^{\sharp_A}}=\widetilde{T}.\\
	(ii)~\|T\|_A=\|\widetilde{T}\|_{\mathcal{B}(\mathbf{R}(A^{1/2}))}, ~ w_A(T)=w(\widetilde{T}) ~ \mbox{and}~~c_A(T)=c(\widetilde{T}).$ 
\end{lemma}

\noindent The following lemma is also obvious from the definition of $\widetilde{T}.$ 

\begin{lemma}\label{lem4}
	Let $S, T\in \mathcal{B}_{A^{1/2}}(\mathcal{H})$ and let $\lambda\in \mathbb{C}$ be any scalar. Then
	\[   \widetilde{S+\lambda T}=\widetilde{S}+\lambda \widetilde{T} ~~\mbox{and}~~ \widetilde{ST}=\widetilde{S}\widetilde{T}.\]
\end{lemma}

\noindent We end this section by elaborating the steps that will be used to  develop the $A$-operator seminorm and the $A$-numerical radius inequalities of semi-Hilbertian space operators.
\begin{itemize}
\item [{ \bf Step 1.}] Begin with $T\in \mathcal{B}_{A^{1/2}}(\mathcal{H}).$ 
\item [{ \bf Step 2.}] Consider the corresponding $\widetilde{T}\in \mathcal{B}(\mathbf{R}(A^{1/2})).$
\item [{ \bf Step 3.}] Look at the classical operator norm and  numerical radius inequalities for the operator $ \widetilde{T}$ acting on the Hilbert space $ \mathbf{R}(A^{1/2}).$
\item [{ \bf Step 4.}] Go back to $A$-operator seminorm and $A$-numerical radius inequalities for the operator $T\in \mathcal{B}_{A^{1/2}}(\mathcal{H}).$
\end{itemize}

\section{Main results}

We begin this section with the  following theorem that gives an $A$-operator seminorm inequality of the product of semi-Hilbertian space operators.
\begin{theorem}\label{th1}
Let $S,T,X\in \mathcal{B}_A(\mathcal{H}).$ Then
\[\|S^{\sharp_A}XT\|_A \leq \frac{1}{2}\|SS^{\sharp_A}X+XTT^{\sharp_A}\|_A. \]
\end{theorem}

\begin{proof}
Since $\mathcal{B}_A(\mathcal{H}) \subseteq \mathcal{B}_{A{^{1/2}}}(\mathcal{H})$, so $S,T,X\in \mathcal{B}_{A{^{1/2}}}(\mathcal{H}).$ Therefore, it follows from Lemma \ref{lem1} that there exists unique $\widetilde{S}$ in $\mathcal{B}(\mathbf{R}(A^{1/2}))$ such that $Z_AS =\widetilde{S}Z_A$. Similarly, there exist $\widetilde{T}$ and $\widetilde{X}$ in $\mathcal{B}(\mathbf{R}(A^{1/2}))$ such that $Z_AT =\widetilde{T}Z_A$ and $Z_AX =\widetilde{X}Z_A$. Following \cite{BD}, we have if $S,T,X\in \mathcal{B}(\mathcal{H})$ then $2\|S^*XT\|\leq \|SS^*X+XTT^*\|.$ 
It follows that 
\begin{eqnarray*}
\|\big(\widetilde{S}\big)^*\widetilde{X} \widetilde{T}\|_{\mathcal{B}(\mathbf{R}(A^{1/2}))} &\leq & \frac{1}{2} \|\widetilde{S}\big(\widetilde{S}\big)^*\widetilde{X}+\widetilde{X}\widetilde{T} \big(\widetilde{T}\big)^*  \|_{\mathcal{B}(\mathbf{R}(A^{1/2}))}\\
\Rightarrow \|\widetilde{S^{\sharp_A}}\widetilde{X} \widetilde{T}\|_{\mathcal{B}(\mathbf{R}(A^{1/2}))} &\leq & \frac{1}{2} \|\widetilde{S}~~\widetilde{S^{\sharp_A}}\widetilde{X}+\widetilde{X}\widetilde{T} ~~\widetilde{T^{\sharp_A}}  \|_{\mathcal{B}(\mathbf{R}(A^{1/2}))},~~\mbox{by Lemma \ref{lem2}(i)}\\
\Rightarrow \|\widetilde{S^{\sharp_A}XT } \|_{\mathcal{B}(\mathbf{R}(A^{1/2}))} &\leq & \frac{1}{2}  \| \widetilde{SS^{\sharp_A}X}+\widetilde{XTT^{\sharp_A}}\|_{\mathcal{B}(\mathbf{R}(A^{1/2}))},~~\mbox{by Lemma \ref{lem4}}\\
\Rightarrow  \|S^{\sharp_A}XT\|_A &\leq& \frac{1}{2} \|SS^{\sharp_A}X+XTT^{\sharp_A}\|_A,~~\mbox{by Lemma \ref{lem2}(ii)}.
\end{eqnarray*}
This completes the proof.
\end{proof}

Considering $X=I$ in the above theorem we get the following corollary.
\begin{cor}\label{cor1}
Let $S,T\in \mathcal{B}_A(\mathcal{H}).$ Then
\[\|S^{\sharp_A}T\|_A \leq \frac{1}{2}\|SS^{\sharp_A}+TT^{\sharp_A}\|_A. \]
\end{cor}

Proceeding in the same way  we can prove the following theorem by using the corresponding results  from  \cite[Th. 2.4, 2.7, 2.17]{BP}  and \cite[Cor. 3.16]{BP}, respectively.

\begin{theorem}\label{th2}
Let $S, T\in \mathcal{B}_A(\mathcal{H}).$ Then the following inequalities hold:
\begin{eqnarray*}
(i)~\|S+T\|_A &\leq & \bigg(\|S\|_A^2+\|T\|_A^2+ \|S^{\sharp_{A}}T+T^{\sharp_{A}}S\|_A \bigg)^{\frac{1}{2}}\leq \|S\|_A+\|T\|_A,\\
(ii)~\|S+T\|_A  &\leq & \bigg(\|S\|_A^2+\|T\|_A^2+\|S\|_A \|T\|_A+\min \big \{ w_A(S^{\sharp_{A}}T), w_A(ST^{\sharp_{A}}) \big \} \bigg)^{\frac{1}{2}}\\
& \leq &\|S\|_A+\|T\|_A,\\
(iii)~\|S+T\|_A^2 &\leq & \|S\|_A^2+\|T\|_A^2+\frac{1}{2}\|S^{\sharp_{A}}S+T^{\sharp_{A}}T\|_A+w_A(S^{\sharp_{A}}T),\\
(iv)~\|S+T\|_A^2 & \leq & \|S\|_A^2+\|T\|_A^2+\frac{1}{2}\|SS^{\sharp_{A}}+TT^{\sharp_{A}}\|_A+w_A(ST^{\sharp_{A}}),
\end{eqnarray*}
and 
\begin{eqnarray*}
(v)~\|ST^{\sharp_{A}}\|_A &\leq & \frac{1}{\sqrt{3}}\bigg\| \left(\frac{S^{\sharp_{A}}S+T^{\sharp_{A}}T}{2}\right)^2 +\|ST^{\sharp_{A}}\|_A^2I + \|ST^{\sharp_{A}}\|_A\left(\frac{S^{\sharp_{A}}S+T^{\sharp_{A}}T}{2}\right) \bigg\|_A^{\frac{1}{2}} \\
&\leq & \frac{1}{2}\left\| S^{\sharp_{A}}S+T^{\sharp_{A}}T \right\|_A.
\end{eqnarray*}
\end{theorem}

\begin{remark}
It follows from Theorem \ref{th2}(i) that if $\|S+T\|_A= \|S\|_A+\|T\|_A$ then $ \|S^{\sharp_{A}}T+T^{\sharp_{A}}S\|_A=2\|S\|_A\|T\|_A.$ Also from  Theorem \ref{th2}(ii) it follows  that if $\|S+T\|_A= \|S\|_A+\|T\|_A$ then $w_A(S^{\sharp_{A}}T)=\|S\|_A\|T\|_A$ and $w_A(ST^{\sharp_{A}})=\|S\|_A\|T\|_A.$

\end{remark}

We next obtain a lower bound for  the $A$-numerical radius that improves on \cite[Th. 1]{kais}.

\begin{theorem}\label{th3.1}
Let $T\in \mathcal{B}_A(\mathcal{H}).$ Then
\begin{eqnarray*}
&&\frac{1}{4}\|T^{\sharp_{A}}T+TT^{\sharp_{A}}\|_A \leq  \frac{1}{8}\bigg( \|T+T^{\sharp_{A}}\|_A^2+\|T-T^{\sharp_{A}}\|_A^2\bigg),~~~\textit{and}\\ 
&& \frac{1}{8}\bigg( \|T+T^{\sharp_{A}}\|_A^2+\|T-T^{\sharp_{A}}\|_A^2\bigg) +\frac{1}{8}c_A^2\big(T+T^{\sharp_{A}}\big)+\frac{1}{8}c_A^2\big(T-T^{\sharp_{A}}\big)
\leq w^2_A(T).
\end{eqnarray*}

\end{theorem}

\begin{proof}
Let $T={\rm Re}_A(T)+ \rm i ~~{\rm Im}_A(T)$, where ${\rm Re}_A(T)=\frac{T+T^{\sharp_A}}{2}$ and ${\rm Im}_A(T)=\frac{T-T^{\sharp_A}}{2\rm i}.$ Then we have that
\begin{eqnarray*}
\frac{1}{4}\|T^{\sharp_{A}}T+TT^{\sharp_{A}}\|_A=\frac{1}{2}\| Re^2_A(T)+Im^2_A(T)\|_A\leq \frac{1}{2}\bigg( \|{\rm Re}_A(T)\|_A^2+\|{\rm Im}_A(T)\|_A^2 \bigg).
\end{eqnarray*}
This implies the first inequality of the theorem.
Now let $x\in \mathcal{H}$ with $\|x\|_A=1.$ Then from the decomposition $T={\rm Re}_A(T)+ \rm i ~~{\rm Im}_A(T)$, we have that
\begin{eqnarray*} 
|\langle Tx,x\rangle_A|^2=|\langle {\rm Re}_A(T)x,x\rangle_A|^2+|\langle {\rm Im}_A(T)x,x\rangle_A|^2.
\end{eqnarray*}
This implies that $$|\langle Tx,x\rangle_A|^2 \geq |\langle {\rm Re}_A(T)x,x\rangle_A|^2+ c_A^2({\rm Im}_A(T)).$$
Since ${\rm Re}_A(T)$ is A-selfadjoint, taking supremum over $\|x\|_A=1$, we get
\begin{eqnarray}\label{3.1.1} 
w^2_A(T)\geq \|{\rm Re}_A(T)\|^2_A+ c_A^2({\rm Im}_A(T)). 
\end{eqnarray}
Similarly, $ {\rm Im}_A(T)$ is A-selfadjoint and we get
\begin{eqnarray}\label{3.1.2} 
w^2_A(T)\geq \|{\rm Im}_A(T)\|^2_A+ c_A^2({\rm Re}_A(T)). 
\end{eqnarray}
Combining (\ref{3.1.1}) and (\ref{3.1.2}), we have
 $$ c_A^2({\rm Im}_A(T))+ c_A^2({\rm Re}_A(T))+ \|{\rm Im}_A(T)\|^2_A+ \|{\rm Re}_A(T)\|^2_A \leq 2w^2_A(T).$$
This implies the second inequality of the theorem, and hence completes the proof.

\end{proof}

\begin{remark}
$(i)$~ It follows from Theorem \ref{th3.1} that if $T\in \mathcal{B}_A(\mathcal{H})$, then
\begin{eqnarray}\label{weak}
	\frac{1}{4}\|T^{\sharp_{A}}T+TT^{\sharp_{A}}\|_A+\frac{1}{8}c_A^2\big(T+T^{\sharp_{A}}\big)+\frac{1}{8}c_A^2\big(T-T^{\sharp_{A}})\leq w^2_A(T).
\end{eqnarray}

$(ii)$~Feki in \cite[Th. 1]{kais} proved that if $T\in \mathcal{B}_A(\mathcal{H})$, then \[\frac{\|T\|_A}{2} \leq \sqrt{\frac{1}{4}\|T\|_A^2+\frac{1}{4}\max \left \{m_A^2(T), m_A^2(T^{\sharp_A})  \right \}}\leq w_A(T),\]
where $m_A(T)=  \inf \{\|Tx\|_A : x\in \mathcal{H}, \|x\|_A= 1\}$.
In \cite[Th. 2.9]{BNP2}, we proved that
	\[\frac{1}{4}\|T\|_A^2+\frac{1}{4}\max \left \{m_A^2(T), m_A^2(T^{\sharp_A})  \right \}\leq \frac{1}{4}\|T^{\sharp_A}T+TT^{\sharp_A}\|_A.\]
Therefore, the first inequality in (\ref{1.2}) is sharper than \cite[Th. 1]{kais}. Clearly, the inequality (\ref{weak}) is sharper than the first inequality in (\ref{1.2}) and so it is sharper than \cite[Th. 1]{kais}. 
\end{remark}

We next prove the following theorem.

\begin{theorem}\label{th3.7}
Let $S,T\in \mathcal{B}_A(\mathcal{H}).$ Then
\[w_A(T^{\sharp_A}S) \leq \frac{1}{2} \|  S^{\sharp_A}S+T^{\sharp_A}T \|_A. \]
\end{theorem}

\begin{proof}
It was proved in \cite{D} that if $S,T\in \mathcal{B}(\mathcal{H})$, then $2w(T^*S)\leq \|S^*S+T^*T\|.$
Since $\mathcal{B}_A(\mathcal{H}) \subseteq \mathcal{B}_{A{^{1/2}}}(\mathcal{H})$, so $S,T\in \mathcal{B}_{A{^{1/2}}}(\mathcal{H}).$ Therefore, it follows from Lemma \ref{lem1} that there exists unique $\widetilde{S}$ in $\mathcal{B}(\mathbf{R}(A^{1/2}))$ such that $Z_AS =\widetilde{S}Z_A$. Similarly, there exists unique  $\widetilde{T}$ in $\mathcal{B}(\mathbf{R}(A^{1/2}))$ such that $Z_AT =\widetilde{T}Z_A$. Therefore, from the above inequality we have that
\begin{eqnarray*}
w\left( \big(\widetilde{T}\big)^*\widetilde{S}  \right) &\leq & \frac{1}{2} \|\big(\widetilde{S}\big)^*\widetilde{S}+ \big(\widetilde{T}\big)^*\widetilde{T}  \|_{\mathcal{B}(\mathbf{R}(A^{1/2}))}\\
\Rightarrow w\left( \widetilde{T^{\sharp_A}}\widetilde{S} \right)  &\leq & \frac{1}{2} \|\widetilde{S^{\sharp_A}}\widetilde{S}+ \widetilde{T^{\sharp_A}} \widetilde{T}  \|_{\mathcal{B}(\mathbf{R}(A^{1/2}))},~~\mbox{by Lemma \ref{lem2}(i)}\\
\Rightarrow w\left( \widetilde{T^{\sharp_A}S } \right) &\leq & \frac{1}{2}  \| \widetilde{S^{\sharp_A}S}+\widetilde{T^{\sharp_A}T}\|_{\mathcal{B}(\mathbf{R}(A^{1/2}))},~~\mbox{by Lemma \ref{lem4}}\\
\Rightarrow w\left( \widetilde{T^{\sharp_A}S } \right) &\leq & \frac{1}{2}  \| \widetilde {S^{\sharp_A}S+T^{\sharp_A}T }\|_{\mathcal{B}(\mathbf{R}(A^{1/2}))},~~\mbox{by Lemma \ref{lem4}}\\
\Rightarrow  w_A\left(T^{\sharp_A}S\right) &\leq& \frac{1}{2} \|S^{\sharp_A}S+T^{\sharp_A}T\|_A,~~\mbox{by Lemma \ref{lem2}(ii)}.
\end{eqnarray*}
This completes the proof.
\end{proof}
Proceeding in the same way we can prove the following theorem by using the corresponding results from \cite[Th. 2.10, 2.13, 2.18]{BP} and \cite[Cor. 2.21, 3.5]{BP},  respectively.

\begin{theorem}\label{th3.2}
Let $T\in \mathcal{B}_A(\mathcal{H}).$ Then the following inequalities hold:
\begin{eqnarray*}
(i)~   w_A^2(T) & \geq & \frac{1}{8}\bigg( \max \big\{ \|T+T^{\sharp_{A}}\|_A^2,\|T-T^{\sharp_{A}}\|_A^2 \big\}+\|T+T^{\sharp_{A}}\|_A\|T-T^{\sharp_{A}}\|_A \bigg), \\
(ii)~w_A^2(T) & \geq & \frac{1}{4\sqrt{2}} \bigg(\big\|T+T^{\sharp_{A}}\big\|_A^4+\big \|T-T^{\sharp_{A}}\big\|_A^4\bigg)^{\frac{1}{2}} \geq \frac{1}{4}\|T^{\sharp_{A}}T+TT^{\sharp_{A}}\|_A,\\
(iii)~ w_A^2(T) & \geq & \frac{1}{8}\left[ \left(\|T+T^{\sharp_{A}}\|_A^2+\|T-T^{\sharp_{A}}\|_A^2\right)^2 + \frac{1}{2}\left(\|T+T^{\sharp_{A}}\|_A^2-\|T-T^{\sharp_{A}}\|_A^2\right)^2\right]^{\frac{1}{2}},
\end{eqnarray*}
\begin{eqnarray*}
(iv)~\frac{1}{2}\|T^{\sharp_{A}}T+TT^{\sharp_{A}}\|_A -  \frac{1}{4}\bigg\|(T+T^{\sharp_{A}})^2 (T-T^{\sharp_{A}})^2 \bigg\|_A^{\frac{1}{2}}\leq w^2_A(T)  \leq 	\frac{1}{2}\|T^{\sharp_{A}}T+TT^{\sharp_{A}}\|_A,
\end{eqnarray*}
and
\begin{eqnarray*}
(v)~w^2_A(T) &\leq &\frac{1}{\sqrt{3}} \left\|  \left(\frac{T^{\sharp_{A}}T+TT^{\sharp_{A}}}{2}\right)^2 +w_A^4(T)I + w_A^2(T)\left(\frac{T^{\sharp_{A}}T+TT^{\sharp_{A}}}{2}\right) \right\|_A^{\frac{1}{2}} \\
& \leq  & \frac{1}{2} \left\| T^{\sharp_{A}}T+TT^{\sharp_{A}}  \right\|_A.
\end{eqnarray*}
\end{theorem}

\begin{remark}
$(a)$ It is clear that 
$$ \|T+T^{\sharp_{A}}\|_A^2+\|T-T^{\sharp_{A}}\|_A^2 \leq \max \big\{ \|T+T^{\sharp_{A}}\|_A^2,\|T-T^{\sharp_{A}}\|_A^2 \big\}+\|T+T^{\sharp_{A}}\|_A\|T-T^{\sharp_{A}}\|_A.$$ Therefore, the inequality in Theorem \ref{th3.2}$(i)$  improves on the first inequality in (\ref{1.2}). \\
$(b)$ By using the first inequality in Theorem \ref{th3.1}, we conclude that the inequality in Theorem \ref{th3.2}$(iii)$ is stronger than the left hand inequality in (\ref{1.2}). \\
$(c)$ It follows from Theorem \ref{th3.2}$(iv)$ that if $\|(T+T^{\sharp_{A}})^2 (T-T^{\sharp_{A}})^2 \|_A=0$, i.e., if $A(T+T^{\sharp_{A}})^2 (T-T^{\sharp_{A}})^2=0$  then $w^2_A(T)  = 	\frac{1}{2}\|T^{\sharp_{A}}T+TT^{\sharp_{A}}\|_A.$\\
$(d)$ Clearly, the inequality in Theorem \ref{th3.2}$(v)$ is stronger than the right hand inequality in (\ref{1.2}). 
\end{remark}

Our next theorem is an improvement of  \cite[Th. 7]{kais01}.

\begin{theorem}\label{th-}
Let $T\in \mathcal{B}_A(\mathcal{H})$. Then
\begin{eqnarray*}
 w_A(T) &\leq &   \frac{1}{\sqrt{2}} \left[\|T^2\|_A^{\frac{1}{2}} \left(  \frac{1}{2}\|T\|_A+ \frac{1}{2}\|T^2\|_A^{\frac{1}{2}}  \right) +\frac{1}{2}\| TT^{\sharp_A}+T^{\sharp_A}T\|_A  \right]^{\frac{1}{2}}\\
& \leq&    \frac{1}{2}\|T\|_A+ \frac{1}{2}\|T^2\|_A^{\frac{1}{2}}. 
\end{eqnarray*}
\end{theorem}

\begin{proof}
Since $\mathcal{B}_A(\mathcal{H}) \subseteq \mathcal{B}_{A{^{1/2}}}(\mathcal{H})$, so $T\in \mathcal{B}_{A{^{1/2}}}(\mathcal{H}).$ Therefore, it follows from Lemma \ref{lem1} that there exists unique $\widetilde{T}$ in $\mathcal{B}(\mathbf{R}(A^{1/2}))$ such that $Z_AT =\widetilde{T}Z_A$. Therefore, following \cite[Cor. 1]{bag} we have
\begin{eqnarray*}
w^2(\widetilde{T}) &\leq & \frac{1}{2} \|\widetilde{T}^2\|^{\frac{1}{2}}_{\mathcal{B}(\mathbf{R}(A^{1/2}))} \left( \frac{1}{2}\|\widetilde{T}\|_{\mathcal{B}(\mathbf{R}(A^{1/2}))}+\frac{1}{2}\|\widetilde{T}^2\|^{\frac{1}{2}}_{\mathcal{B}(\mathbf{R}(A^{1/2}))}   \right)\\
&& +\frac{1}{4}\|(\widetilde{T})^*\widetilde{T}+\widetilde{T}(\widetilde{T})^*\|_{\mathcal{B}(\mathbf{R}(A^{1/2}))}. 
\end{eqnarray*}
Now by using Lemma \ref{lem4} and Lemma \ref{lem2} we get the first inequality of the theorem. The second inequality follows from \cite[Remark 5]{bag} by using similer arguments as above.
\end{proof}

\begin{remark}
Feki in \cite[Th. 7]{kais01} proved that $$w_A(T) \leq \frac{1}{2}\|T\|_A+ \frac{1}{2}\|T^2\|_A^{\frac{1}{2}}. $$ Clearly, the first inequality in Theorem \ref{th-} is sharper than that in \cite[Th. 7]{kais01}.

\end{remark}

The final theorem of this paper is  a well-known result \cite{BPN, MXZ}, which is given here to highlight the advantage of the technique used in this paper.

\begin{theorem}\label{th-B&M}
Let $T\in \mathcal{B}_A(\mathcal{H})$. Then
\begin{eqnarray*}
	\frac{1}{16}\|TT^{\sharp_A}+T^{\sharp_A}T\|_A^2+	\frac{1}{4}c_A\left(\big({\rm Re}_A(T^2)\big)^2\right) &\leq & w^4_A(T)\\
	&\leq & \frac{1}{8}\| TT^{\sharp_A}+T^{\sharp_A}T\|^2_A+\frac{1}{2}w^2_A(T^2). 
\end{eqnarray*}
\end{theorem}

\begin{proof}
Since $\mathcal{B}_A(\mathcal{H}) \subseteq \mathcal{B}_{A{^{1/2}}}(\mathcal{H})$, so $T\in \mathcal{B}_{A{^{1/2}}}(\mathcal{H}).$ Therefore, it follows from Lemma \ref{lem1} that there exists unique $\widetilde{T}$ in $\mathcal{B}(\mathbf{R}(A^{1/2}))$ such that $Z_AT =\widetilde{T}Z_A$. It follows from \cite[Th. 8]{bag} that 
\begin{eqnarray*}
	\frac{1}{16}\|(\widetilde{T})^*\widetilde{T}+\widetilde{T}(\widetilde{T})^*\|_{\mathcal{B}(\mathbf{R}(A^{1/2}))}^2+	\frac{1}{4}c\left(\big({\rm Re}_A(\widetilde{T}^2)\big)^2\right) &\leq & w^4(\widetilde{T})
\end{eqnarray*}
and 
\begin{eqnarray*}	
w^4(\widetilde{T}) &\leq & \frac{1}{8}\|(\widetilde{T})^*\widetilde{T}+\widetilde{T}(\widetilde{T})^*\|_{\mathcal{B}(\mathbf{R}(A^{1/2}))}^2+\frac{1}{2}w^2(\widetilde{T}^2). 
\end{eqnarray*}
Thus, the required inequalities follow from  Lemma \ref{lem4} and Lemma \ref{lem2}.
\end{proof}

\begin{remark}
The proof  of the inequalities given here is simple and short in comparison to the proofs  of the same in  \cite{BPN, MXZ}. We conclude that the same holds true for many existing proofs on $A$-operator seminorm and $A$-numerical radius inequalities for semi-Hilbertian space operators.

\end{remark}

\bibliographystyle{amsplain}

\begin{thebibliography}{99}

\bibitem{acg3} {M.L. Arias, G. Corach and M.C. Gonzalez,} {Lifting properties in operator ranges,} Acta Sci. Math. (Szeged), 75:3-4(2009), 635-653.


\bibitem{Arias1} M.L. Arias, G. Corach and  M.C. Gonzalez, Partial isometries in semi-Hilbertian spaces, Linear Algebra Appl. 428 (2008) 1460-1475.

\bibitem{bag} S. Bag, P. Bhunia  and K. Paul, Bounds of numerical radius of bounded linear operator using $t$-Aluthge transform, Math. Inequal. Appl.  23(3) (2020) 991-1004.	



\bibitem{BD} R. Bhatia and C. Davis, More matrix forms of the arithmetic-geometric mean inequality, SIAM J. Matrix Anal. Appl., 14(1993) 132-136.


\bibitem{BPN} P. Bhunia, K. Paul and   R.K. Nayak, On inequalities for A-numerical radius of operators, Electron. J. Linear Algebra 36 (2020) 143-157.


\bibitem{BNP} P. Bhunia, R.K. Nayak and   K. Paul,  Refinements of A-numerical radius inequalities and their applications, Adv. Oper. Theory 5 (2020) 1498-1511.


\bibitem{BP2} P. Bhunia and   K. Paul,  Some  improvement of numerical radius inequalities of operators and operator matrices, Linear Multilinear Algebra (2020). \url{https://doi.org/10.1080/03081087.2020.1781037}


\bibitem{BFP} P. Bhunia, K. Feki and  K. Paul,  $A$-numerical radius orthogonality and parallelism of semi-Hilbertian space operators and their applications, Bull. Iran. Math. Soc. (2020). \url{https://doi.org/10.1007/s41980-020-00392-8}


\bibitem{BP} P. Bhunia and K. Paul, Refinements of norm and numerical radius inequalities, (2020).
\url{arXiv:2010.12750v1[math.FA]}


\bibitem{BNP2} P. Bhunia, R.K. Nayak and K. Paul, Improvement of $A$-numerical radius inequalities of semi-Hilbertian space operators, (2020), arXiv:2008.10840v3 [math.FA].	


\bibitem{branrov} L. de Branges and J. Rovnyak, Square summable power series, Holt, Rinehert and Winston, New York, 1966.


\bibitem{doug} R.G. Douglas, On majorization, factorization and range inclusion of operators in Hilbert space, Proc. Amer. Math. Soc. 17 (1966) 413-416.

\bibitem{D} S.S. Dragomir, Power inequalities for the numerical radius of a product of two operators in Hilbert spaces, Sarajevo J. Math. 5(18) (2009) 269-278.


\bibitem{kais01} K. Feki, Spectral radius of semi-Hilbertian space operators and its applications. Ann. Funct. Anal., (2020). \url{https://doi.org/10.1007/s43034-020-00064-y}

\bibitem{kais} 	K. Feki, A note on the A-numerical radius of operators in semi-Hilbert spaces, Arch. Math. (2020). \url{https://doi.org/10.1007/s00013-020-01482-z}



\bibitem{majsecesuci} W. Majdak, N.A. Secelean and L. Suciu, Ergodic properties of operators in some semi-Hilbertian spaces, Linear Multilinear Algebra
61(2) (2013) 139-159.

\bibitem{MXZ} M.S. Moslehian, Q. Xu and A. Zamani, Seminorm and numerical radius inequalities of operators in semi-Hilbertian spaces, Linear Algebra Appl. 591 (2020) 299-321.	

\bibitem{Z} A. Zamani, A-numerical radius inequalities for semi-Hilbertian space operators, Linear Algebra Appl. 578 (2019) 159-183.	


\end{thebibliography}

\end{document}